\documentclass[12pt,a4paper]{article}
\usepackage[utf8]{inputenc}
\usepackage[T1]{fontenc}
\usepackage{a4wide}

\usepackage[figuresright]{rotating}
\usepackage{amssymb}
\usepackage{amsmath}
\usepackage{amsthm}
\usepackage{amsfonts}
\usepackage{graphics}
\usepackage{enumerate}
\usepackage{tikz}
\usepackage{caption}
\usepackage{subcaption}
\usepackage{float}
\usepackage{comment}
\usepackage{hyperref}

\newtheorem{theorem}{Theorem}
\newtheorem{corollary}[theorem]{Corollary}
\newtheorem{lemma}[theorem]{Lemma}

\newtheorem{conjecture}[theorem]{Conjecture}
\newtheorem{claim}[theorem]{Claim}
\theoremstyle{definition}

\newcommand\pd{\bar{\delta}^\pm}

\title{Antidirected paths in oriented graphs\thanks{This work was supported by the National Science Centre grant 2021/42/E/ST1/00193.}}

\author{
Andrzej Grzesik\thanks{Faculty of Mathematics and Computer Science, Jagiellonian University, {\L}ojasiewicza 6, 30-348 Krak\'{o}w, Poland. E-mail: {\tt andrzej.grzesik@uj.edu.pl}.}
\and
Marek Skrzypczyk\thanks{Faculty of Mathematics and Computer Science, Jagiellonian University, {\L}ojasiewicza 6, 30-348 Krak\'{o}w, Poland. E-mail: {\tt marek.skrzypczyk@doctoral.uj.edu.pl}.}
}

\date{}

\begin{document}

\maketitle

\begin{abstract}
We show that for any integer $k \ge 4$, every oriented graph with minimum semidegree bigger than $\frac{1}{2}(k-1+\sqrt{k-3})$ contains an antidirected path of length~$k$. Consequently, every oriented graph on $n$ vertices with more than $(k-1+\sqrt{k-3})n$ edges contains an antidirected path of length $k$. This 
asymptotically proves the antidirected path version of a conjecture of Stein and of a conjecture of Addario-Berry, Havet, Linhares Sales, Reed and Thomassé, respectively. 
\end{abstract}

\section{Introduction}
The classic results of Dirac \cite{dirac} and of Erd\H os and Gallai \cite{erdosgallai} determine optimal bounds on the minimum degree and on the number of edges that forces a graph to contain a path of a given length. In this paper, we continue recent line of research on similar conditions forcing an oriented graph to contain a path oriented in alternating directions. This is the most natural and interesting orientation of a path, for which no optimal bounds are known. 

An \emph{oriented graph} $G$ is a graph with an orientation on every edge. In particular, for every $u, v \in V(G)$ at most one of $uv$ and $vu$ is an edge of $G$. If there is an edge $uv$ in $E(G)$, then we say that it is directed from $u$ to $v$, $v$ is an \emph{out-neighbour} of $u$ and $u$ is an \emph{in-neighbour} of $v$. The \emph{in-degree} of $v$, denoted by $d^-(v)$, is the number of in-neighbours of~$v$, while the \emph{out-degree} of $v$, denoted by $d^+(v)$, is the number of out-neighbours of $v$. 
The \emph{minimum semidegree} of an oriented graph $G$ is $\delta^\pm(G)=\min\{d^-(v),d^+(v)\,:\, v\in V(G)\}$. We also define the \emph{minimum pseudo-semidegree} $\pd(G)$ of $G$ as follows. If $G$ has no edges, then $\pd(G)=0$, otherwise $\pd(G)$ is the largest integer $d$ such that each vertex has out-degree either 0 or at least $d$, and in-degree either 0 or at least $d$. Obviously always $\pd(G) \ge \delta^\pm(G)$. 

An \emph{antidirected path}, for short \emph{antipath} (\emph{antidirected cycle}, in short \emph{anticycle}), is an orientation of a path (cycle) such that for each vertex $v$ we have either $d^-(v)=0$ or $d^+(v)=0$. The \emph{length} of an oriented path (cycle) is the number of its edges. Note that there is only one antipath of a given odd length, whereas there are two non-isomorphic antipaths of any given even length. Clearly, each anticycle has even length.

Motivated by Jackson's result \cite{Jackson} for a directed path, Stein \cite{stein_tree} conjectured the following condition on the minimum semidegree forcing an oriented graph to contain any given orientation of a path. 

\begin{conjecture}[Stein \cite{stein_tree}]\label{stein}
For any integer $k \geq 1$, every oriented graph $G$ with $\delta^\pm(G) > \frac{1}{2}k$ contains each oriented path of length $k$.
\end{conjecture}

The above threshold, if true, is tight. This can be seen by noticing that a blow-up of any directed cycle with $\frac{1}{2}k$ vertices in each blob has semidegree equal to $\frac{1}{2}k$ and does not contain an antipath of length $k$. 
The conjecture is known to be true for directed paths~\cite{Jackson} and for $k \le 4$~\cite{antipaths}. 
Since the antipath is the most natural orientation for which Conjecture~\ref{stein} is open, and known extremal examples witness that not containing an antipath seems to be the most restrictive condition, attention was recently drawn to the antipaths. 
In particular, Stein and Zárate-Guerén \cite{antidirected} proved that an approximate version of Conjecture~\ref{stein} holds for antipaths in large oriented graphs on $n$ vertices if $k$ is linear in $n$. 

An interesting observation is that if we replace the minimum semidegree condition in Conjecture~\ref{stein} for antipaths by the weaker condition on the minimum pseudo-semidegree, then we obtain an equivalent conjecture. It is so, because one can combine copies of an oriented graph $G$ with its copies having reversed orientations on edges to obtain an oriented graph~$G'$ with $\delta^\pm(G') = \pd(G)$ without creating longer antipaths. Therefore, Conjecture~\ref{stein} for antipaths is equivalent to the following problem. 

\begin{conjecture}\label{pseudostein}
For any integer $k \geq 1$, every oriented graph $G$ with $\pd(G) > \frac{1}{2}k$ contains each antipath of length $k$.
\end{conjecture}

Klimo\v{s}ov\'{a} and Stein \cite{antipaths} proved that for $k \geq 3$ the stronger bound $\pd(G) > \frac{3}{4}k-\frac{3}{4}$ implies that an oriented graph $G$ contains each antipath of length $k$. This was further improved by Chen, Hou and Zhou~\cite{chen} to approximately $\frac{2}{3}k$ and by Skokan and Tyomkyn~\cite{skokantyomkyn} to $\frac{5}{8}k$. 

In this paper, we improve the minimum pseudo-semidegree bound to a value that agrees with the conjectured $\frac{1}{2}k$ threshold up to a $\mathcal{O}(\sqrt{k})$ error term. 

\begin{theorem}\label{maintheorem}
For any integer $k\ge 4$, every oriented graph $G$ with $\pd(G) > \frac{1}{2}(k-1+\sqrt{k-3})$ contains each antipath of length $k$. 
\end{theorem}

As noted in \cite{antipaths} and \cite{stein_tree}, Conjecture~\ref{pseudostein} implies analogous results in related important conjectures on forcing antipaths in oriented graphs. 

In 2013 Addario-Berry, Havet, Linhares Sales, Thomassé and Reed \cite{ADDARIO} stated the following conjecture. 

\begin{conjecture}[Addario-Berry et al. \cite{ADDARIO}]\label{hip.addario}
For any integer $k \geq 1$, every oriented graph~$G$ on $n$ vertices with more than $(k-1)n$ edges contains each antidirected tree with $k$ edges.
\end{conjecture}

Stein and Zárate-Guerén \cite{antidirected} showed that any oriented graph on $n$ vertices with more than $cn$ edges contains an oriented graph $G$ with $\pd(G) > \frac{1}{2}c$. This means that any improvement in the bound for the minimum pseudo-semidegree forcing an oriented graph to contain antipaths is providing an improvement on the bound for the edge number forcing to contain antipaths. Therefore, Theorem~\ref{maintheorem} implies the following corollary asymptotically solving Conjecture~\ref{hip.addario} for antipaths. 

\begin{corollary}\label{cor:addario}
For any integer $k \geq 4$, every oriented graph $G$ on $n$ vertices with more than $(k-1+\sqrt{k-3})n$ edges contains each antipath of length $k$.
\end{corollary}

Conjecture~\ref{hip.addario} was motivated by a more general conjecture of Burr \cite{burr} from 1980. We say that an oriented graph is \emph{$m$-chromatic} if its underlying graph has chromatic number~$m$. 

\begin{conjecture}[Burr \cite{burr}]\label{con:burr}
Every $2k$-chromatic oriented graph contains each oriented tree with $k$ edges.
\end{conjecture}

Since every graph with chromatic number $m$ contains a subgraph with minimum degree at least $m-1$, and consequently with many edges,  Corollary~\ref{cor:addario} asymptotically solves Conjecture~\ref{con:burr} for antipaths. 

\begin{corollary}
Every $(2k+2\sqrt{k-3})$-chromatic oriented graph contains each antipath of length $k$.
\end{corollary}

\section{Auxiliary lemmas}

First, we recall the following auxiliary lemmas proved by Klimo\v{s}ov\'{a} and Stein in \cite{antipaths}.

\begin{lemma}[Klimo\v{s}ov\'{a}, Stein \cite{antipaths}]\label{lemma1}
Let $k \ge 1$ be an integer and $G$ be an oriented graph with $\pd(G)\ge \frac{1}{2}k$ having a longest antipath of length $m$. If $m<k$, then $m$ is odd.
\end{lemma}

\begin{lemma}[Klimo\v{s}ov\'{a}, Stein \cite{antipaths}]\label{lemma3}
Let $k \geq 1$ be an integer and $G$ be an oriented graph with $\pd(G) > \frac{1}{2}k$ containing an anticycle of length $m+1$. If $m<k$, then $G$ contains an antipath of length $m+1$.
\end{lemma}

The next auxiliary lemma will allow to modify the start or the end of a longest antipath. 

\begin{lemma}\label{lem:Fsize}
Let $k \geq 1$ be an integer and $G$ be an oriented graph with $\pd(G) \ge \frac{1}{2}k$ having a longest antipath of length $m$. If $m<k$, then $G$ contains an antipath $A = v_0v_1\ldots v_m$ with $v_0v_1 \in E(G)$ such that $v_1$ has at least $\pd(G)-\frac{1}{2}k$ in-neighbours in $V(G)\setminus V(A)$, or $v_{m-1}$ has at least $\pd(G)-\frac{1}{2}k$ out-neighbours in $V(G)\setminus V(A)$.
\end{lemma}

\begin{proof}
By Lemma~\ref{lemma1} $m$ is odd. Let $A = v_0v_1\ldots v_m$ be a longest antipath in $G$. By symmetry, we may assume that $v_0v_1 \in E(G)$. Denote by $E=\{v_0,v_2,\ldots,v_{m-1}\}$ the set of vertices in $V(A)$ with even index and by $O=\{v_1,v_3,\ldots,v_m\}$ the set of vertices in $V(A)$ with odd index. Note that every edge of $A$ is directed from $E$ to $O$. 

If $v_0 v_{2i+1} \in E(G)$ for some positive $i < \frac{1}{2}m$, then $v_{2i} v_{2i-1}\ldots v_0 v_{2i+1} v_{2i+2}\ldots v_m$ is an antipath on all vertices of $A$ with all edges directed from $E$ to $O$. Similarly, if $v_{2i} v_{m} \in E(G)$, then also $v_0v_1\ldots v_{2i}v_mv_{m-1}\ldots v_{2i+1}$ is such an antipath. Let $S \subseteq O$ and $P \subseteq E$ be the sets of vertices that can be the second and the penultimate vertices in some antipath on all vertices of $A$ that can be obtained from $A$ by the above operations. In particular, $v_1 \in S$ and $v_{m-1} \in P$. 

\begin{claim}\label{cla:FPlubSL}
There is an edge from $v_0$ to $P$ or from $S$ to $v_m$. 
\end{claim}

\begin{proof}
For the sake of contradiction assume that there are no such edges. This means that vertex $v_0$ has at most $\frac{m-1}{2} - |P|$ out-neighbours in $E$, so at least $\pd(G) - \frac{m-1}{2} + |P|$ out-neighbours in $O$. If $v_{2i+1}$ is an out-neighbour of $v_0$ in $O$ for $i > 0$, then $v_{2i-1} \in S$. Therefore, $|S| \ge \pd(G) - \frac{m-1}{2} - 1 + |P| > |P|$. A symmetric argument from the perspective of $v_m$ leads to $|P| > |S|$, which gives a contradiction. 
\end{proof}

By Claim~\ref{cla:FPlubSL} and symmetry, we may assume that $G$ contains an edge $v_1v_m$. This means that vertex $v_1$ has simultaneously positive in-degree and out-degree. In particular, it has at least $2\pd(G)-m+1 \ge 2\pd(G)-k+2$ neighbours in the set $R = V(G) \setminus \{v_2,v_3,\ldots, v_m\}$. Thus, at least  $\pd(G)-\frac{1}{2}k+1$ in-neighbours or at least  $\pd(G)-\frac{1}{2}k+1$ out-neighbours in~$R$. 
In the former case, $A$ satisfies the conclusion of the lemma. 
In the later case, note that each out-neighbour $w \in R$ of $v_1$ creates an antipath $v_2v_3\ldots v_mv_1w$ of length $m$ with $v_1$ as the penultimate vertex having at least $\pd(G)-\frac{1}{2}k$ out-neighbours outside this antipath (see Figure~\ref{f1}), so the conclusion of the lemma also holds.
\end{proof}

\begin{figure}[ht]
\centering
\begin{tikzpicture}[x=10cm,y=10cm] 
  \tikzset{     
    e4c node/.style={circle,fill=black,minimum size=0.2cm,inner sep=0}, 
    e4c edge/.style={line width=1.5pt}
  }
  \node[e4c node] (f) at (0.00, 0.80) {}; 
  \node[anchor= north] at (0.00, 0.80) {$v_0$};
  \node[e4c node] (l) at (0.00, 0.65) {}; 
  \node[anchor= north] at (0.00, 0.65) {$w$};
  \node[e4c node] (v1) at (0.15, 0.80) {}; 
  \node[anchor= north west] at (0.15, 0.80) {$v_1$};
  \node[e4c node] (v2) at (0.30, 0.80) {};
  \node[anchor= north] at (0.30, 0.80) {$v_2$};
  \node[e4c node] (v3) at (0.45, 0.80) {}; 
  \node[anchor= north] at (0.45, 0.80) {$v_3$};
  \node[e4c node] (vm-1) at (0.85, 0.80) {}; 
  \node[anchor= north] at (0.85, 0.80) {$v_{m-1}$};
  \node[e4c node] (vm) at (1.00, 0.80) {}; 
  \node[anchor= north] at (1.00, 0.80) {$v_m$};

  \path[->,draw, line width=2pt, >=stealth, shorten >=1pt]
  (f) edge[e4c edge]  (v1)
  (v2) edge[e4c edge]  (v1)
  (v2) edge[e4c edge]  (v3)
  (vm-1) edge[e4c edge]  (vm)
  (v1) edge[e4c edge]  (l)
  ;
  
  \path[->,draw,dotted,line width=2pt, >=stealth, shorten >=1pt]
  (vm-1) edge[e4c edge] (v3)
  ;
  
  \path[->,draw,thick, bend left, >=stealth,shorten >=1pt]
  (v1) edge[e4c edge]  (vm)
  ;
\end{tikzpicture}

\caption{Two types of antipaths created in Lemma~\ref{lem:Fsize}.}\label{f1}
\end{figure}
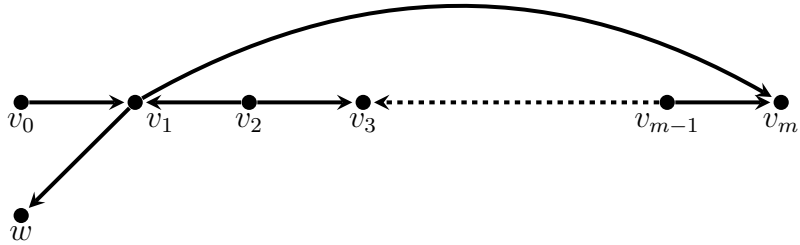

\section{Main proof}

\begin{proof}[Proof of Theorem~\ref{maintheorem}]
Consider any integer $k \ge 4$ and an oriented graph $G$ satisfying $\pd(G) > \frac{1}{2}(k-1+\sqrt{k-3})$ not containing an antipath of length $k$. By Lemma~\ref{lem:Fsize} and symmetry, let $A = v_0v_1v_2\ldots v_m$ be a longest antipath in $G$ with $v_0v_1 \in E(G)$ and $v_1$ having at least $\pd(G)-\frac{1}{2}k$ in-neighbours in $V(G)\setminus V(A)$.  

Let $F$ be the set of at least $\pd(G)-\frac{1}{2}k+1 \geq 2$ in-neighbours of $v_1$ that are not in the set $\{v_2, v_3, \ldots, v_m\}$. Note that each vertex $w$ in $F$ is the first vertex on an antipath $wv_1v_2\ldots v_m$ of length $m$, in particular $v_0 \in F$. 

\begin{claim}\label{cla:forbiddenpairs}
For any different vertices $v, w \in F$ and a positive integer $i < \frac{1}{2}m$, if $vv_{2i} \in E(G)$ and $wv_{2i+1} \in E(G)$, then $wv_{2i-1} \not\in E(G)$ and $wv_{2i} \not\in E(G)$.
\end{claim}

\begin{proof}
The claim follows from the maximality of the antipath $A$. 
If $wv_{2i-1} \in E(G)$, then $v_{2i}vv_1v_2\ldots v_{2i-1}wv_{2i+1}v_{2i+2}\ldots v_m$ is a longer antipath (see Figure \ref{figure2}), while if $wv_{2i} \in E(G)$, then $v_{2i-1}v_{2i-2}\ldots v_1 vv_{2i}wv_{2i+1}v_{2i+2}\ldots v_m$ is a longer antipath (see Figure \ref{figure3}).
\end{proof}

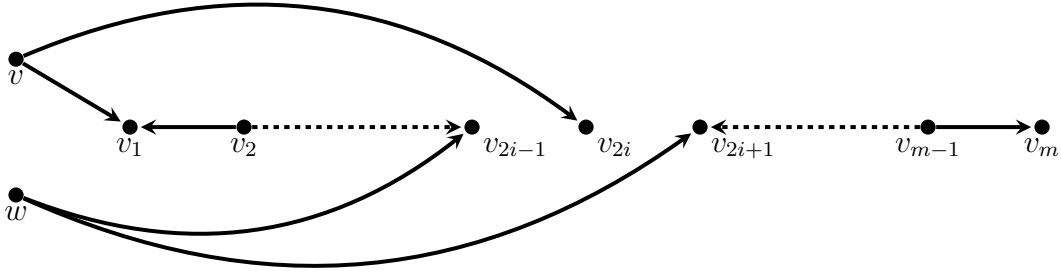
\begin{figure}[h]
\centering
\begin{tikzpicture}[x=10cm,y=9cm] 
  \tikzset{     
    e4c node/.style={circle,fill=black,minimum size=0.2cm,inner sep=0}, 
    e4c edge/.style={line width=1.5pt}
  }
  \node[e4c node] (f1) at (0.00, 0.90) {}; 
  \node[anchor= north] at (0.00, 0.90) {$v$};
  \node[e4c node] (f2) at (0.00, 0.70) {}; 
  \node[anchor= north] at (0.00, 0.70) {$w$};
  \node[e4c node] (v1) at (0.15, 0.80) {}; 
  \node[anchor= north] at (0.15, 0.80) {$v_1$};
  \node[e4c node] (v2) at (0.30, 0.80) {};
  \node[anchor= north] at (0.30, 0.80) {$v_2$};
  \node[e4c node] (vm-1) at (1.20, 0.80) {}; 
  \node[anchor= north] at (1.20, 0.80) {$v_{m-1}$};
  \node[e4c node] (vm) at (1.35, 0.80) {}; 
  \node[anchor= north] at (1.35, 0.80) {$v_m$};
  
   \node[e4c node] (v2i-1) at (0.60, 0.80) {}; 
  \node[anchor= north west] at (0.60, 0.80) {$v_{2i-1}$};
  \node[e4c node] (v2i) at (0.75, 0.80) {}; 
  \node[anchor= north west] at (0.75, 0.80) {$v_{2i}$};
  \node[e4c node] (v2i+1) at (0.90, 0.80) {}; 
  \node[anchor= north west] at (0.90, 0.80) {$v_{2i+1}$};

  \path[->,draw, line width=2pt, >=stealth, shorten >=1pt]
  (f1) edge[e4c edge]  (v1)
  (v2) edge[e4c edge]  (v1)
  (vm-1) edge[e4c edge]  (vm)
  ;
  
  \path[->,draw,dotted,line width=2pt, >=stealth, shorten >=1pt]
  (v2) edge[e4c edge] (v2i-1)
  (vm-1) edge[e4c edge] (v2i+1)
  ;
  
 \path[->,draw,thick, bend left, >=stealth,shorten >=1pt]
(f1) edge[e4c edge]  (v2i)
  ;
  
  \path[->,draw,thick, bend right, >=stealth,shorten >=1pt]
(f2) edge[e4c edge]  (v2i-1)
(f2) edge[e4c edge]  (v2i+1)
;
\end{tikzpicture}

\caption{The antipath of length $m+1$ if $wv_{2i-1} \in E(G)$.} \label{figure2}
\end{figure}

\begin{figure}[h]
\centering
\begin{tikzpicture}[x=10cm,y=9cm] 
  \tikzset{     
    e4c node/.style={circle,fill=black,minimum size=0.2cm,inner sep=0}, 
    e4c edge/.style={line width=1.5pt}
  }
  \node[e4c node] (f1) at (0.00, 0.90) {}; 
  \node[anchor= north] at (0.00, 0.90) {$v$};
  \node[e4c node] (f2) at (0.00, 0.70) {}; 
  \node[anchor= north] at (0.00, 0.70) {$w$};
  \node[e4c node] (v1) at (0.15, 0.80) {}; 
  \node[anchor= north] at (0.15, 0.80) {$v_1$};
  \node[e4c node] (v2) at (0.30, 0.80) {};
  \node[anchor= north] at (0.30, 0.80) {$v_2$};
  \node[e4c node] (vm-1) at (1.20, 0.80) {}; 
  \node[anchor= north] at (1.20, 0.80) {$v_{m-1}$};
  \node[e4c node] (vm) at (1.35, 0.80) {}; 
  \node[anchor= north] at (1.35, 0.80) {$v_m$};
  
   \node[e4c node] (v2i-1) at (0.60, 0.80) {}; 
  \node[anchor= north] at (0.60, 0.80) {$v_{2i-1}$};
  \node[e4c node] (v2i) at (0.75, 0.80) {}; 
  \node[anchor= north west] at (0.75, 0.80) {$v_{2i}$};
  \node[e4c node] (v2i+1) at (0.90, 0.80) {}; 
  \node[anchor= north west] at (0.90, 0.80) {$v_{2i+1}$};
  
  \path[->,draw, line width=2pt, >=stealth, shorten >=1pt]
  (f1) edge[e4c edge]  (v1)
  (v2) edge[e4c edge]  (v1)
  (vm-1) edge[e4c edge]  (vm)
  ;
  
  \path[->,draw,dotted,line width=2pt, >=stealth, shorten >=1pt]
  (v2) edge[e4c edge] (v2i-1)
  (vm-1) edge[e4c edge] (v2i+1)
  ;
  
 \path[->,draw,thick, bend left, >=stealth,shorten >=1pt]
(f1) edge[e4c edge]  (v2i)
  ;
  
  \path[->,draw,thick, bend right, >=stealth,shorten >=1pt]
(f2) edge[e4c edge]  (v2i)
(f2) edge[e4c edge]  (v2i+1)
;
\end{tikzpicture}

\caption{The antipath of length $m+1$ if $wv_{2i} \in E(G)$.} \label{figure3}
\end{figure}
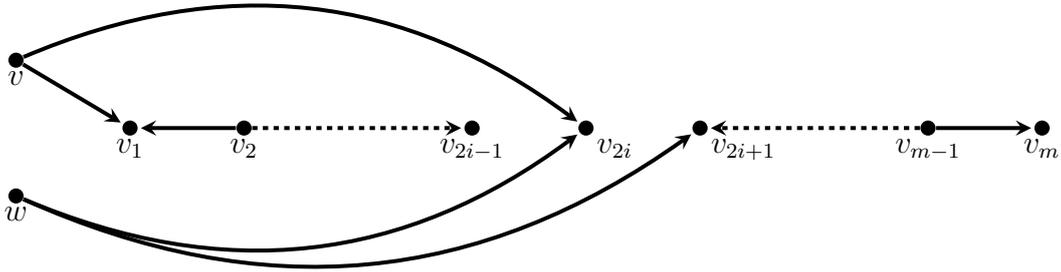

\begin{claim}\label{cla:edgesfromF}
For every integer $i \leq \frac{m-3}{4}$ there are at most $2|F|+1$ edges from $F$ to the set $\{v_{4i}, v_{4i+1}, v_{4i+2}, v_{4i+3}\}$, and at most $2|F|$ edges from $F$ to the set $\{v_{0}, v_{1}, v_{2}, v_{3}\}$.
\end{claim}

\begin{proof}
Let $Q = \{v_{4i}, v_{4i+1}, v_{4i+2}, v_{4i+3}\}$. If there is a vertex in $F$ with $4$ out-neighbours in $Q$, then by Claim~\ref{cla:forbiddenpairs} any other vertex in $F$ does not have an edge to $v_{4i}$, $v_{4i+2}$ and simultaneously to $v_{4i+1}$ and $v_{4i+3}$, so there are at most $|F|+3 \leq 2|F|+1$ edges from $F$ to~$Q$. If two different vertices $v, w \in F$ have $3$ out-neighbours in $Q$, then by Claim~\ref{cla:forbiddenpairs} either $vv_{4i+2} \not\in E(G)$ or $wv_{4i+3} \not\in E(G)$, similarly, either $wv_{4i+2} \not\in E(G)$ or $vv_{4i+3} \not\in E(G)$. This implies that  $v$ and $w$ have edges to both $v_{4i}$ and $v_{4i+1}$, which also contradicts  Claim~\ref{cla:forbiddenpairs}. Therefore, there is at most $1$ vertex in $F$ with $3$ out-neighbours in $Q$ and all the other vertices in $F$ have at most $2$ out-neighbours in $Q$. This implies that there are at most $2|F|+1$ edges from $F$ to $Q$.

To prove the second part of the claim, note that $v_0$ is not an out-neighbour of any vertex in $F$ as otherwise it creates an antipath of length $m+1$ contradicting the maximality of $m$. If there is a vertex in $F$ with edges to both $v_{2}$ and $v_{3}$, then by Claim~\ref{cla:forbiddenpairs} any other vertex in $F$ has at most 1 out-neighbour in the set $\{v_{0}, v_{1}, v_{2}, v_{3}\}$, so there are at most $|F|+2 \leq 2|F|$ edges from $F$ to $\{v_{0}, v_{1}, v_{2}, v_{3}\}$. Otherwise, each vertex in $F$ has at most $2$ out-neighbours in the set $\{v_{0}, v_{1}, v_{2}, v_{3}\}$ as wanted. 
\end{proof}

Note that there are no edges from $F$ to $v_m$ as otherwise it creates an anticycle of length $m+1$ and by Lemma~\ref{lemma3} we obtain a contradiction with the maximality of $m$. Hence, there are at most $|F|$ edges from $F$ to the set $\{v_{m-1},v_m\}$. 

Therefore, Claim~\ref{cla:edgesfromF} implies that if $m \equiv 3$ mod $4$, then there are at most
\[\frac{m-3}{4}(2|F| + 1) + 2|F| \leq \frac{k}{2}|F| + \frac{k-4}{4}\]
edges from $F$ to $V(A)$, while if $m \equiv 1$ mod $4$, then there are at most 
\[\frac{m-5}{4}(2|F| + 1) + 2|F| + |F| \leq \frac{k}{2}|F| + \frac{k-6}{4} \leq \frac{k}{2}|F| + \frac{k-4}{4}\]
edges from $F$ to $V(A)$. 
On the other hand, since all out-neighbours of the vertices in $F$ are in $V(A)$ as otherwise we have a longer antipath in $G$, there are at least $\pd(G) |F|$ edges from $F$ to $V(A)$. This means that 
\begin{align*}
\frac{k-4}{4} &\geq \left(\pd(G) - \frac{k}{2}\right)|F| \geq \left(\pd(G) - \frac{k}{2}\right)\left(\pd(G)-\frac{k}{2}+1\right) \\
&> \frac{\sqrt{k-3}-1}{2}\cdot\frac{\sqrt{k-3}+1}{2} = \frac{k-4}{4},
\end{align*}
which is a contradiction concluding the proof of Theorem~\ref{maintheorem}.
\end{proof}

\section{Concluding remarks}

The minimum pseudo-semidegree bound in Theorem~\ref{maintheorem} can be further improved by lowering the constant by the $\sqrt{k}$ term. This can be achieved using a more detailed analysis on the structure of edges between $F$ and $A$ and by improving the bound of the size of the set $F$ obtained in Lemma~\ref{lem:Fsize}.

We remark that applying Lemma~\ref{lemma1} and considering proper rounding to the nearest integer in Theorem~\ref{maintheorem} gives that Conjecture~\ref{pseudostein} is true for $k \leq 11$. 

\bibliographystyle{plain}
\bibliography{grafy}

\begin{thebibliography}{10}

\bibitem{ADDARIO}
L.~Addario-Berry, F.~Havet, L.~{Linhares Sales}, B.~Reed, and S.~Thomassé.
\newblock Oriented trees in digraphs.
\newblock {\em Discrete Mathematics}, 313:967--974, 2013.

\bibitem{burr}
S.~A. Burr.
\newblock Subtrees of directed graphs and hypergraphs.
\newblock In {\em Proceedings of the Eleventh Southeastern Conference on
  Combinatorics, Graph Theory and Computing, Boca Raton, Congr. Numer},
  volume~28, pages 227--239, 1980.

\bibitem{chen}
B.~Chen, X.~Hou, and H.~Zhou.
\newblock Long antipaths and anticycles in oriented graphs.
\newblock {\em Discrete Mathematics}, 348:114412, 2025.

\bibitem{dirac}
G.~A. Dirac.
\newblock Some theorems on abstract graphs.
\newblock {\em Proceedings of The London Mathematical Society}, 2:69--81, 1952.

\bibitem{erdosgallai}
P.~Erd{\H o}s and T.~Gallai.
\newblock On maximal paths and circuits of graphs.
\newblock {\em Acta Mathematica Academiae Scientiarum Hungarica}, 10:337--356,
  1959.

\bibitem{Jackson}
B.~Jackson.
\newblock Long paths and cycles in oriented graphs.
\newblock {\em Journal of Graph Theory}, 5:145--157, 1981.

\bibitem{antipaths}
T.~Klimošová and M.~Stein.
\newblock Antipaths in oriented graphs.
\newblock {\em Discrete Mathematics}, 346:113515, 2023.

\bibitem{skokantyomkyn}
J.~Skokan and M.~Tyomkyn.
\newblock Alternating paths in oriented graphs with large semidegree.
\newblock {\em arXiv:2406.03166}, 2024.

\bibitem{stein_tree}
M.~Stein.
\newblock Tree containment and degree conditions.
\newblock {\em Discrete Mathematics and Applications}, 165:459--486, 2020.

\bibitem{antidirected}
M.~Stein and C.~Zárate-Guerén.
\newblock Antidirected subgraphs of oriented graphs.
\newblock {\em Combinatorics, Probability and Computing}, 33:1--21, 2024.

\end{thebibliography}

\end{document}